\documentclass[12pt]{amsart}
\usepackage{enumerate}
\usepackage{graphicx}
\usepackage{amssymb}

\newtheorem{theorem}{Theorem}[section]
\newtheorem{lemma}[theorem]{Lemma}
\newtheorem{definition}[theorem]{Definition}

\newtheorem{remark}[theorem]{Remark}

\newcommand{\beqa}{\begin{eqnarray*}}
\newcommand{\eeqa}{\end{eqnarray*}}
\newcommand{\beqn}{\begin{eqnarray}}
\newcommand{\eeqn}{\end{eqnarray}}
\newcommand{\e}{\varepsilon}

\newcommand{\Ra}{\Rightarrow}
\newcommand{\Lra}{\Leftrightarrow}

\newcommand{\ds}{\displaystyle}
\newcommand{\R}{\mathbb R}

\newcounter{cnt1}
\newcounter{cnt2}
\newcounter{cnt3}
\newcounter{cnt4}
\newcommand{\blr}{\begin{list}{$($\roman{cnt1}$)$}
 {\usecounter{cnt1} \setlength{\topsep}{0pt}
 \setlength{\itemsep}{0pt}}}
\newcommand{\bla}{\begin{list}{$(\alph{cnt2})$}
 {\usecounter{cnt2} \setlength{\topsep}{0pt}
 \setlength{\itemsep}{0pt}}}
\newcommand{\bln}{\begin{list}{$($\arabic{cnt3}$)$}
 {\usecounter{cnt3} \setlength{\topsep}{0pt}
 \setlength{\itemsep}{0pt}}}
\newcommand{\blR}{\begin{list}{$($\Roman{cnt4}$)$}
 {\usecounter{cnt4} \setlength{\topsep}{0pt}
 \setlength{\itemsep}{0pt}}}
\newcommand{\el}{\end{list}}

\begin{document}

\title{On uniform Mazur intersection property}

\author[Bandyopadhyay]{Pradipta Bandyopadhyay}
\address[Pradipta Bandyopadhyay]{Stat--Math Division,
Indian Statistical Institute, 203, B.~T. Road, Kolkata
700108, India.}
\email{pradipta@isical.ac.in}

\author[Ganesh]{Jadav Ganesh}
\address[Jadav Ganesh]{Dept of Mathematics, SLABS, SRM University, Neerukonda, Mangalagiri Mandal, Guntur Dist (AP), India-522502}
\email{ganesh.j@srmap.edu.in}

\author[Gothwal]{Deepak Gothwal}
\address[Deepak Gothwal]{Stat--Math Division, Indian Statistical Institute, 203, B.~T. Road, Kolkata
700108, India.}
\email{deepakgothwal190496@gmail.com}

\subjclass[2010]{46B20 ~\hfill \emph{To appear Studia Math.}}


\dedicatory{Dedicated to Professor Bor-Luh (Peter) Lin on his 85th birthday!}


\begin{abstract}
In this paper, we show that a Banach space $X$ has the Uniform Mazur Intersection Property (UMIP) if and only if every $f \in S(X^*)$ is uniformly w*-semidenting point of $B(X^*)$. We also prove analogous results for uniform w*-MIP.
\end{abstract}

\maketitle

\section{Introduction}

The Mazur Intersection Property (MIP) --- every closed bounded convex set is the intersection of closed balls containing it --- is an extremely well studied property in Banach space theory. MathSciNet search with keywords ``Mazur('s) Intersection Property'' returns 53 hits. The paper \cite{GJM} is an excellent survey of the MIP, its generalisations and variants. In particular, one should mention \cite{GGS}, where a complete characterisation was obtained, most well-known one stating that the w*-denting points of $B(X^*)$ are norm dense in $S(X^*)$. They also considered the property in dual spaces that every w*-compact convex set is the intersection of balls (w*-MIP). In \cite{CL}, Chen and Lin introduced the notion of w*-semidenting points and showed that a Banach space $X$ has the MIP if and only if every $f \in S(X^*)$ is a w*-semidenting point of $B(X^*)$. Among subsequent papers, we must mention~\cite{Gi}.

A much less studied uniform version of the MIP (UMIP or UI) was introduced by Whitfield and Zizler \cite{WZ}. MathSciNet finds only one citation to \cite{WZ}, namely, \cite{GJM}, in the last 30+ years. Characterisations similar to \cite{GGS} were also obtained, but an analogue of the w*-denting point characterisation was missing, which perhaps is a reason for its being less pursued.

In this paper, we show that a Banach space $X$ has the UMIP if and only if every $f \in S(X^*)$ is uniformly w*-semidenting point of $B(X^*)$, thus filling a long felt gap. In the process, we also present simpler proofs of some characterisations in \cite{WZ}. We also introduce a w*-version of UMIP in the spirit of \cite{GGS} and obtain similar characterisations.

We hope that our result will also be a small step towards answering the long standing open question whether the UMIP implies that the space is super-reflexive.

\section{Characterisation of UMIP}

We consider real Banach spaces only. Let $X$ be a Banach space.
For $x\in X$ and $r>0$, we denote by $B(x, r)$ \emph{the open ball} $\{y\in X : \|x-y\|<r\}$ and by $B[x, r]$ \emph{the closed ball} $\{y\in X: \|x-y\|\leq r\}$. We denote by $B(X)$ the \emph{closed unit ball} $\{x\in X : \|x\| \leq 1\}$ and by $S(X)$ the \emph{unit sphere} $\{x \in X : \|x\| = 1\}$.

For $x\in S(X)$, we denote by $D(x)$ the set $\{f\in S(X^*) :
f(x)=1\}$. Any selection of $D$ is called a support mapping.

For a bounded set $C \subseteq X$ and $x \in X$, let \bla
\item $d(x, C) = \inf\{\|x-z\| : z \in C\}$, the distance function; and
\item $diam(C) = \sup\{\|x-y\| : x, y \in C\}$, the diameter of $C$.
\el

\begin{definition} \cite{WZ} \rm
We say that a Banach space $X$ has the Uniform Mazur Intersection Property (UMIP) if for every $\e >0$ and $M(\e) \geq 2$, there is $K(\e) > 0$ such that whenever a closed convex set $C\subseteq X$ and a point $p\in X$ are such that $diam(C) \leq M(\e)$ and $d(p, C) \geq \e$, there is a closed ball $B \subseteq X$ of radius $\leq K(\e)$ such that $C \subseteq B$ and $d(p, B) \geq \e/2$.
\end{definition}

\begin{remark} \rm
In the original definition in \cite{WZ}, $M(\e) = 1/\e$.
\end{remark}

\begin{definition} \rm
A \emph{slice} of $B(X)$ determined by $f\in S(X^*)$ is a set of the form
\[
S(B(X), f, \delta) := \{x\in B(X) : f(x) > 1-\delta \}
\]
for some $0<\delta < 1$. For $x \in S(X)$, $S(B(X^*), x, \delta)$ is called a w*-slice of $B(X^*)$.

We say that $x\in S(X)$ is a \emph{denting point} of $B(X)$ if
for every $\e > 0$, $x$ is contained in a slice of $B(X)$
of diameter less than $\e$.

A w*-denting point of $B(X^*)$ is defined similarly.
\end{definition}

For w*-semidenting points, we use an equivalent definition.
\begin{definition} \rm
$f\in S(X^*)$ is said to be a \emph{w*-semidenting point} of $B(X^*)$ if for every $\e > 0$, there exists a w*-slice
$S(B(X^*), x, \delta) \subseteq B(f, \e)$.

A semidenting point of $B(X)$ can be defined similarly.
\end{definition}

We begin with a variant of Phelps' Parallel Hyperplane Lemma \cite[Lemma 2.1]{Gi}. We include the well-known proof as we will use a step of the proof later.

\begin{lemma} \label{lem2}
For a normed linear space $X$, if for $f, g \in S(X^*)$ and $\e > 0$, $\{x \in B(X) : f(x) > \e\} \subseteq \{x \in X : g(x) > 0\}$, then $\|f-g\| < 2\e$.
\end{lemma}

\begin{proof}
By the given condition, if $x \in B(X)$ and $g(x) = 0$, then
$|f(x)| \leq \e$. That is $\|f|_{\ker(g)}\| \leq \e$. By
Hahn-Banach Theorem, there exists $h \in X^*$ such that $\|h\|
\leq \e$ and $h \equiv f$ on $\ker(g)$. It follows that $f-h = tg$ for some $t \in \R$. Then $\|f-tg\| = \|h\| \leq \e$. Now, if $y \in \{x \in B(X) : f(x) > \e\}$, then $g(y) > 0$ and
\[
\e < f(y) \leq (f-tg)(y) + t g(y) \leq \|f-tg\| + tg(y) \leq
\e + tg(y).
\]
It follows that $t > 0$. And
\[|1-t| \leq \big|\|f\|-\|tg\|\big| \leq \|f-tg\| \leq \e.\]
Thus,
\[\|f-g\| \leq \|f-tg\|+|1-t| \leq \e + \e = 2\e.\]
\end{proof}

The theorem below is equivalent to \cite[Theorem 2.1]{CL1} with a much simpler proof.

\begin{theorem} \label{thm1}
Let $X$ be a Banach space. Let $A \subseteq X^{**}$ be a bounded set. Then there exists a closed ball $B^{**} \subseteq X^{**}$ with centre in $X$ such that $A \subseteq B^{**}$ and $0 \notin B^{**}$ if and only if $d (0, A) > 0$ and there is a w*-slice of $B(X^*)$ contained in
\[
\{f \in B(X^*) : x^{**}(f) > 0 \mbox{ for all } x^{**} \in A\}.
\]
\end{theorem}

\begin{proof}
Let $A \subseteq B^{**}[x_0, r]$ and $0 \notin B^{**}[x_0, r]$. Then $\|x_0\| > r$. Clearly, $d(0, A) \geq \|x_0\| - r > 0$.

Let $S = \{f \in B(X^*) : f(x_0) > r\}$. Then $S$ is a w*-slice of $B(X^*)$. And if $g \in S$, then for any $x^{**} \in A$,
\[g(x_0 - x^{**}) \leq \|x_0 - x^{**}\| \leq r\]
and hence,
\[x^{**}(g) \geq  g(x_0) - r > 0.\]
Thus,
\[
S \subseteq \{f \in B(X^*) : x^{**}(f) > 0 \mbox{ for all } x^{**} \in A\}.
\]

Conversely, let $d = d(0, A) > 0$ and let $x_0 \in S_X$ and
$0 < \e < 1$ be such that
\[
\{f \in B(X^*) : f(x_0) > \e\} \subseteq \{f \in B(X^*) : x^{**}(f) > 0 \mbox{ for all } x^{**} \in A\}.
\]
Let $M = \sup \{\|x^{**}\| : x^{**} \in A\}$. By the proof of Lemma~\ref{lem2}, for all $x^{**} \in A$, there exists $t \in \R$ such that $1-\e \leq t \leq 1+\e$ and
\[
\left\|\frac{t x^{**}}{\|x^{**}\|} - x_0\right\| \leq \e.
\]
Then for $\lambda \geq M/(1-\e)$,
\beqa
\|x^{**}- \lambda x_0\| & \leq & \left\|x^{**}-  \frac{\|x^{**}\|}{t} x_0\right\| + \left|\frac{\|x^{**}\|}{t} - \lambda\right| \leq \frac{\e \|x^{**}\|}{t} + \lambda - \frac{\|x^{**}\|}{t}\\
& = & \lambda - \frac{\|x^{**}\|}{t}(1-\e) \leq \lambda - \frac{d(1-\e)}{1+\e}.
\eeqa
Therefore,
\[A \subseteq B^{**}\left[\lambda x_0, \lambda - \frac{d(1-\e)}{1+\e}\right] \mbox{ and clearly, }
0 \notin B^{**}\left[\lambda x_0, \lambda - \frac{d(1-\e)}{1+\e}\right].
\]
\end{proof}

\begin{remark} \rm
From Theorem~\ref{thm1}, it can be shown that if every $f \in S(X^*)$ is a w*-semidenting point of $B(X^*)$ then $X$ has the MIP. See Theorem~\ref{thm2} $(b) \Ra (a)$ below.
\end{remark}

\begin{definition} \rm
For $\e, \delta>0$ and $ x\in S(X) $, denote by
\beqa
d_1(x, \delta) & = & \sup \limits_{0< \lambda < \delta, \ y \in B(X)} \frac{\|{x+\lambda y}\| +\|{x-\lambda y}\|- 2}{\lambda}, \\
d_2(x, \delta) & = & diam(S(B(X^*), x, \delta)),\\
d_3(x, \delta) & = & diam (D(S(X)\cap B (x, \delta))).
\eeqa
\end{definition}

\begin{definition} \cite{WZ} \rm For $\e, \delta > 0$, define
\[
M_{\e, \delta}(X) = \{x\in S(X) :  \sup_{0< \|y\| < \delta} \frac{\|{x + y}\| + \|{x - y}\|- 2}{\|y\|} < \e\}.
\]
In other words, $M_{\e, \delta}(X) = \{x\in S(X) : d_1(x, \delta )< \e\}$.
\end{definition}

The lemma below is quantitatively more precise than \cite[Lemma 1]{WZ}. It follows from \cite[Lemma 2.1]{Ba}. We include the details for completeness.

\begin{lemma} \label{lem1}
For any $\alpha, \delta > 0$, we have \blr 
\item $\ds d_2(x, \alpha )\leq d_1(x, \delta )+\frac{2\alpha}{\delta}$
\item $d_3(x, \delta)\leq d_2(x, \delta )$
\item $d_1(x, \delta)\leq d_3(x, 2{\delta})$.
\el
\end{lemma}

\begin{proof}
$(i)$. Fix $\alpha$, $\delta>0$. Let $d_2 = d_2 (x, \alpha)$. For each $n \geq 1$, we can choose $f_n$, $g_n \in S(B(X^*), x, \alpha)$ such that \mbox{$\|f_n - g_n\| > d_2 - \frac{1}{n}$.} Choose $y_n \in B(X)$ such that $(f_n - g_n)(y_n) > d_2 - \frac{1}{n}$. Then
\[
\frac{\|x + \delta y_n\| + \|x - \delta y_n\| - 2}{\delta} \geq \frac{f_n(x + \delta y_n) + g_n(x - \delta y_n) - 2}{\delta} \geq d_2 - \frac{1}{n} - \frac{2\alpha}{\delta}.
\]
Thus, $d_1 (x, \delta)\geq d_2 - \frac{2\alpha}{\delta}$.

$(ii)$. Let $y \in S(X) \cap B(x, \delta)$ and $f_{y}\in D(y)$. Then,
\[0 \leq 1 - f_y(x) = f_y(y-x) \leq \|y-x\| < \delta.\]
Thus, $f_y \in S(B(X^*), x, \delta)$.
Therefore, $D(S(X)\cap B(x, \delta)) \subseteq S(B(X^*), x, \delta)$.

$(iii)$. Let $\lambda \leq \delta$. Observe that for any $y \in B(X)$,
\beqa
\left\|\frac{x \pm \lambda y}{\|x \pm \lambda y\|} - x \right\| & \leq & \left\|\frac{x \pm \lambda y}{\|x \pm \lambda y\|} - (x \pm \lambda y) \right\| + \lambda = \big|1 - \|x \pm \lambda y\|\big| + \lambda \\
& = & \big|\|x\|-\|x \pm \lambda y\|\big| + \lambda \leq 2\lambda.
\eeqa

Let $\ds f \in D\left(\frac{x + \lambda y}{\|x + \lambda y\|} \right)$ and $\ds g \in D\left(\frac{x - \lambda y}{\|x - \lambda y\|}\right)$. Then
\beqa
\frac{\|x + \lambda y\| + \|x - \lambda y\| - 2}{\lambda} &=& \frac{f(x + \lambda y) + g(x - \lambda y) - 2}{\lambda} \\ &\leq & (f-g)(y) \leq \|f-g\|.
\eeqa
\end{proof}

We now come to our main theorem. The conditions $(c)$--$(e)$ are reformulations of $(ii)$--$(iv)$ of \cite[Theorem 1]{WZ} in the language of \cite[Theorem 2.1]{GGS}. We give a self-contained proof of $(a) \Lra (b)$, our main result. The proofs of $(c) \Ra (e)$ and $(c) \Lra (d)$ are essentially from \cite[Theorem 1]{WZ}, which we include for the sake of completeness and we will need them in the next section.
For the rest, our arguments are either new or simpler.

\begin{theorem} \label{thm2}
For a Banach space $X$, the following are equivalent~:
\bla
\item $X$ has the UMIP.
\item Every $f \in S(X^*)$ is uniformly w*-semidenting, i.e., given $\e > 0$, there exists $0 < \delta < 1$ such that for any $f \in S(X^*)$, there exists $x \in S(X)$ such that
    \[
    S(B(X^*), x, \delta) \subseteq B(f, \e).
    \]
\item The duality map is uniformly quasicontinuous, i.e., given $\e > 0$, there exists $\delta > 0$ such that for any $f \in S(X^*)$, there exists $x \in S(X)$ such that $D(S(X) \cap B(x, \delta)) \subseteq B(f, \e)$.
\item Given $\e >0 $ there exists $\delta >0$ such that every support mapping maps a $\delta$-net in $S(X)$ to an $\e$-net in $S(X^*)$.
\item Given $\e >0 $ there exists $\delta >0$ such that for every $f\in S(X^*)$, there exists $x \in M_{\e, \delta}(X)$ such that $D(x) \subseteq B(f, \e)$.
\el
\end{theorem}

\begin{proof}
$(a) \Ra (b)$. We use an idea from \cite[Lemma 2.2$(ii)$]{Gi}.

Let $0< \e < 1$ and $M(\e) \geq 2$ be given. Choose $K = K(\e)$ as given by $(a)$ for $\e/3$. We may assume w.l.o.g.\ that $K \geq 1$.

Let $f \in S(X^*)$. Consider $C := \{x\in B(X) : f(x) \geq \e/3\}$. Then, $d(0, C) \geq \e/3$. Also, $diam (C) \leq 2 \leq M(\e)$.

So, we have $B = B[x_0, r]$ containing $C$, $r \leq K$ and $d(0, B) \geq \e/6$.

Since $d(0, B[x_0, r]) \geq \e/6$, $\|x_0\| \geq r + \e/6 > r + \e/9$.

\textsc{Claim~:} $S =: S(B(X^*), x_0/\|x_0\|, 1 - K/(K+\e/9)) \subseteq B(f, \e)$, that is, $(b)$ holds for $x = x_0/\|x_0\|$ and $\delta = 1 - K/(K+\e/9)$.

Let $g \in S$. Since, $r/(r+\e/9) \leq K/(K+\e/9)$, $g(x_0/\|x_0\|) > r/(r+\e/9)$, which implies $g(x_0) > r\|x_0\|/(r+\e/9) > r$. That is, $\inf g(B) > 0$. So,
\[
\{x\in B(X) : f(x) > \e/3\} \subseteq C \subseteq B \subseteq \{x: g(x) > 0\}.
\]
By Lemma~\ref{lem2}, $\|f - g/\|g\|\| < 2\e/3$. Also, $g \in S$ implies $\|g\| \geq g(x_0/\|x_0\|) > K/(K+\e/9)$. Hence,
\[\|f - g\| \leq \|f- g/\|g\|\| + (1 - \|g\|) \leq 2\e/3 + \e/(9K+\e) < \e.
\]

This proves the claim.

$(b) \Ra (a)$.
Given $\e >0$ and $M(\e) \geq 2$, it suffices to show that there is $K(\e) > 0$ such that whenever a closed convex set $C \subseteq X$ is such that $diam(C) \leq M(\e)$ and $d(0, C) \geq \e$, there is a closed ball $B \subseteq X$ of radius $\leq K(\e)$ such that $C \subseteq B$ and $d(0, B) \geq \e/2$.

Let $\e>0$ be given. Let $L = M(\e) + \e$. Choose $0< \delta < 1$ for $\e/4L$ obtained from $(b)$. Let $K = L/\delta + 1$. We will show that this $K$ works.

\textsc{Case I~:} $C \setminus B[0, M(\e) + \e/2] \neq \emptyset$.

Choose $z \in C \setminus B[0, M(\e) + \e/2]$. Then $C \subseteq B[z, M(\e)]$, $d(0, B[z, M(\e)]) \geq \e/2$ and $M(\e) \leq K$.

\textsc{Case II~:} $C \subseteq B[0, M(\e) + \e/2]$.

Define $D = \overline{C + \frac{\e}{2} B(X)}$. Then $D \subseteq B[0, L]$ and $d(0, D) \geq \e/2$, and hence, $D$ is disjoint from $B(0, \e/2)$. By separation theorem, there exists $f \in S(X^*)$ such that $\inf f(D) \geq \e/2$. By choice of $\delta$, there exists $x_0 \in S(X)$ such that
\[
S(B(X^*), x_0, \delta) \subseteq B(f, \e/4L).
\]

It follows that if $g \in B(X^*) \cap B(f, \e/4L)$ and $z \in D$, then $g(z) \geq f(z) - \|f-g\| \|z\| \geq \e/2 - \e/4 = \e/4 > 0$. Therefore,
\beqa
S(B(X^*), x_0, \delta) & \subseteq & B(X^*) \cap B(f, \e/4L)\\ & \subseteq & \{g \in B(X^*) : g(x) > 0 \mbox{ for all } x \in D\}.
\eeqa

By the proof of Theorem~\ref{thm1}, for $K \geq \lambda \geq L/\delta$,
\[
D \subseteq B\left[\lambda x_0, \lambda - \frac{d(0, D)\delta}{2-\delta}\right].
\]

It follows that
\beqa
& \ds C \subseteq B = B\left[\lambda x_0, \lambda - \frac{\e}{2} - \frac{d(0, D)\delta}{2-\delta}\right], \quad \mbox{and}\\
& \ds d(0, B) = \frac{\e}{2} + \frac{d(0, D)\delta}{2-\delta} \geq \frac{\e}{2}.
\eeqa

This completes the proof.

$(b) \Ra (c)$. As observed in the proof of Lemma~\ref{lem1}$(ii)$, $D(S(X)\cap B(x, \delta)) \subseteq S(B(X^*), x, \delta)$.

$(c) \Ra (e)$. Let $\e >0 $ be given. Then from $(c)$ for $\e/3$, there exists $\delta > 0$ such that for any $f \in S(X^*)$, there exists $x \in S(X)$ such that $D(S(X) \cap B(x, 2\delta)) \subseteq B(f, \e/3)$. This implies $d_3(x, 2\delta) \leq 2\e/3 < \e$. By Lemma~\ref{lem1}$(iii)$, $d_1(x, \delta) < \e$, \emph{i.e.}, $x \in M_{\e, \delta}(X)$. And $D(x) \subseteq B(f, \e/3) \subseteq B(f, \e)$.

$(e) \Ra (b)$. Let $\e > 0$ be given. By $(e)$, there exists $\delta >0$ such that for every $f\in S(X^*)$, there is an $x\in M_{\e/4, \delta}(X)$ such that $D(x)\subseteq B(f, \e/4)$. So, $d_1(x, \delta) < \e/4$. By Lemma~\ref{lem1}$(i)$, $d_2(x, \alpha) < \e/2$, for $\alpha <\delta \e/8$. That is, $diam(S(B(X^*), x, \alpha)) < \e/2$. Now, $D(x) \subseteq S(B(X^*), x, \alpha)$. Hence for any $g \in S(B(X^*), x, \alpha)$ and $f_x \in D(x)$
\[
\|f-g\| \leq \|f - f_x\| + \|f_x - g\| \leq \e/4 + \e/2 < \e.
\]
Therefore,
\[
S(B(X^*), x, \alpha) \subseteq B(f, \e).
\]

$(c) \Ra (d)$ is obvious.

$(d) \Ra (c)$. If $(c)$ does not hold, there exists $\e > 0$ such that for all $\delta > 0$, there exists $f_\delta \in S(X^*)$ such that for all $x \in S(X)$, there exists $z_x \in S(X)$ and $f_{z_x} \in D(z_x)$ such that $\|z_x - x\| < \delta/2$ and $\|f_\delta -f_{z_x}\| \geq \e$. Then $\{z_x : x \in S(X)\}$ is a $\delta$-net, but $\{f_{z_x}: x \in S(X)\}$ is not an $\e$-net in $S(X^*)$.
\end{proof}

\section{Characterisation of w*-UMIP}

In this section, we study the w*-version of UMIP defined above.

\begin{definition} \rm
We say that a dual Banach space $X^*$ has the w*-Uniform Mazur Intersection Property (w*-UMIP) if for every $\e >0$ and $M(\e) \geq 2$, there is $K(\e) > 0$ such that whenever a w*-compact convex set $C\subseteq X^*$ and a point $f\in X^*$ are such that $diam(C) \leq M(\e)$ and $d(f, C) \geq \e$, there is a closed ball $B \subseteq X^*$ of radius $\leq K(\e)$ such that $C \subseteq B$ and $d(f, B) \geq \e/2$.
\end{definition}

For $f \in S(X^*)$, the inverse duality mapping is defined as
\[
D^{-1}(f) =
\left\{ \begin{array}{ll} \{x \in S(X) : f(x)=1\} & \mbox{ if \quad} f \in D(S(X)) \\ \emptyset & \mbox{ if \quad} f \notin D(S(X)) \end{array} \right.
\]
The fact that $D^{-1}(f) = \emptyset$ for some $f \in S(X^*)$ introduces some technicalities in the discussion of w*-UMIP.

\begin{definition} \rm
For $\e, \delta >0$ and $f \in S(X^*)$, denote by
\beqa
d^*_1(f, \delta) & = & \sup \limits_{0< \lambda < \delta, ~g \in B(X^*)} \frac{\|{f + \lambda g}\| + \|{f - \lambda g}\|- 2}{\lambda},\\
d^*_2(f, \delta) & = & diam(S(B(X), f, \delta)),\\
d^*_3(f, \delta) & = & diam (D^{-1}(D(S(X)) \cap B (f, \delta))).
\eeqa
\end{definition}

\begin{lemma} \label{lem3}
For any $\alpha, \delta > 0$, we have,
\begin{enumerate}[(i)]
\item $\ds d^*_2(f, \alpha) \leq d^*_1(f, \delta ) + \frac{2\alpha}{\delta}$
\item $d^*_3(f, \delta) \leq d^*_2(f, \delta)$
\item $\ds d^*_1(f, \alpha)\leq \frac{d^*_3(f, \delta)+2\alpha}{1-\alpha^2}$ if $\alpha < \sqrt{\delta+1} - 1$.
\end{enumerate}
\end{lemma}

\begin{proof}
$(i)$ follows from Lemma~\ref{lem1}$(i)$.

$(ii)$ holds since
$D^{-1}(D(S(X)) \cap B(f, \delta)) \subseteq S(B(X), f, \delta)$.

$(iii)$. Let $d^*_3 = d^*_3(f, \delta)$. Choose $0< \alpha < \sqrt{\delta+1} - 1$. Then $\alpha^2 + 2\alpha < \delta$. Let $g \in S(X^*)$, $0 < \lambda < \alpha$. As before,
\[
\left\|\frac{f \pm \lambda g}{\|f \pm \lambda g\|} - f \right\| \leq 2 \lambda.
\]

Find $h_1, h_2 \in D(S(X))$ such that
\[
\left\|\frac{f + \lambda g}{\|f + \lambda g\|} - h_1 \right\| \leq \lambda^2 \quad \mbox{and} \quad
\left\|\frac{f - \lambda g}{\|f - \lambda g\|} - h_2 \right\| \leq \lambda^2
\]
and $x_1, x_2 \in S(X)$ such that $h_i \in D(x_i)$, $i = 1, 2$.

Observe that $\|h_i-f\| \leq \lambda^2 + 2\lambda \leq \alpha^2+2\alpha <\delta$, i.e., $h_1, h_2 \in D(S(X)) \cap B(f, \delta)$. Thus $\|x_1-x_2\| \leq d^*_3$. Now,
\[
0 \leq 1-x_1\left(\frac{f+\lambda g}{\|f+\lambda g\|}\right) = x_1 \left(h_1 - \frac{f + \lambda g}{\|f + \lambda g\|}\right) \leq \left\|\frac{f + \lambda g}{\|f + \lambda g\|} - h_1\right\| \leq \lambda^2.
\]
So, $ x_1(f + \lambda g) \geq (1-\lambda^2)\|f + \lambda g\|$. Similarly, $x_2(f - \lambda g) \geq (1- \lambda^2)\|f - \lambda g\|$. So, we have
\beqa
\frac{\|f + \lambda g\| + \|f - \lambda g\| -2}{\lambda} & \leq & \frac{x_1 (f + \lambda g)+ x_2(f - \lambda g) - 2 (1-\lambda^2)}{\lambda (1-\lambda^2)} \\
& = & \frac{(x_1 + x_2)(f) - 2 + \lambda(x_1-x_2)(g) + 2 \lambda^2}{\lambda (1-\lambda^2)}\\
& \leq & \frac{\|x_1-x_2\| + 2\lambda}{1-\lambda^2} \leq \frac{d^*_3 + 2\lambda}{1-\lambda^2} \leq \frac{d^*_3+2\alpha}{1-\alpha^2}
\eeqa
(since $\ds \frac{d^*_3+2\lambda}{1-\lambda^2}$ is increasing in $\lambda$). Thus,
\[d^*_1(f, \alpha) \leq \frac{d^*_3+2\alpha}{1-\alpha^2}.\]
\end{proof}

\begin{remark} \rm
The proof of $(iii)$ is adapted from the proof of \cite[Lemma 2]{Ba2}. Notice that unlike the proof of \cite[Lemma 3.1]{GGS}, our proof uses just the Bishop-Phelps Theorem and not Bollob\'as' extensions.
\end{remark}

We now come to our characterisation theorem. Though the proof is similar to that of Theorem~\ref{thm2}, there are added technicalities as pointed out earlier.

\begin{theorem}
For a Banach space $X$, the following statements are equivalent~:
\bla
\item $X^*$ has w*-UMIP.
\item Every $x \in S(X)$ is uniformly semidenting, i.e., given $\e >0$, there exists $\delta >0$ such that for any $x \in S(X)$, there exists $f \in S(X^*)$ such that
    \[
    S(B(X), f, \delta) \subseteq B(x, \e).
    \]
\item The inverse duality map is uniformly quasicontinuous, i.e., given $ \e >0$, there exists $\delta >0$ such that for any $x \in S(X)$, there exists $f \in D(S(X))$ such that $D^{-1}( D(S(X)) \cap B(f, \delta)) \subseteq B(x,\e)$.
\item Given $\e >0 $ there exists $\delta >0$ such that every support mapping on $X^*$ that maps $D(S(X))$ into $S(X)$ maps any $\Delta \subseteq D(S(X))$ that is a $\delta$-net in $S(X^*)$ to an $\e$-net in $S(X)$.
\item Given $\e >0 $ there exists $\delta >0$ such that for every $x \in S(X)$, there exists $f \in M_{\e, \delta}(X^*) \cap D(S(X))$ such that $D^{-1}(f) \subseteq B(x, \e)$.
\el
\end{theorem}

\begin{proof}
$(a) \Ra (b)$ is similar to Theorem~\ref{thm2} $(a) \Ra (b)$.

$(b) \Ra (a)$.
Define $D = C + \frac{\e}{2} B(X^*)$. Then $D$ is w*-compact, and $d(0, D) \geq \e/2$. Therefore, $D$ is disjoint from the w*-compact convex set $3\e/8 B(X^*)$. Hence there exists $x \in S(X)$ such that $\inf \hat{x}(D) \geq 3\e/8$.

Rest of the proof is similar to Theorem~\ref{thm2} $(b) \Ra (a)$.

$(b) \Ra (c)$. Let $\e >0$ be given. By $(b)$, there exists $\delta >0$ such that for any $x \in S(X)$, there exists $g \in S(X^*)$ such that
\[
S(B(X), g, 2\delta) \subseteq B(x, \e).
\]

As in \cite[Lemma 3.2.6]{B}, if we choose $0 < \eta < \delta/4$ then for any $h \in B(g, \eta)$,
\[
S(B(X), h, \delta) \subseteq S(B(X), g, 2\delta).
\]

By Bishop-Phelps Theorem, there exists $f \in D(S(X)) \cap B(g, \eta)$. Then
\[
S(B(X), f, \delta) \subseteq S(B(X), g, 2\delta) \subseteq B(x, \e).
\]

And $D^{-1}(D(S(X)) \cap B(f, \delta)) \subseteq S(B(X), f, \delta)$. Therefore, $(c)$ holds.

$(c) \Ra (e)$. Let $\e >0 $ be given. Then from $(c)$ for $\e/4$, there exists $\delta >0$ such that for any $x \in S(X)$, there exists $f \in D(S(X))$ such that $D^{-1}(D(S(X)) \cap B(f, \delta)) \subseteq B(x,\e/4)$. This implies $d^*_3(f, \delta) \leq \e/2$. By Lemma~\ref{lem3}$(iii)$,
\beqa
d^*_1(f, \alpha) & \leq & \frac{d^*_3(f, \delta) + 2\alpha} {1-\alpha^2} \leq \frac{\e/2 + 2\alpha} {1-\alpha^2} < \e \\ && \mbox{ if } \alpha < \sqrt{\delta+1} - 1 \mbox{ and } \frac{\e/2 + 2\alpha} {1-\alpha^2} < \e.
\eeqa

That is, $f \in M_{\e, \alpha}(X^*)$. Clearly, the choice of $\alpha$ depends only on $\e$. And $D^{-1}(f) \subseteq B(x, \e/4) \subseteq B(x, \e)$.

$(e)\Ra (b)$.  Let $\e >0$ be given. Choose $0 <\delta <1$ obtained from (e) for $\e/4$. So, given $x \in S(X)$, there is $f \in M_{\e/4, \delta}(X^*) \cap D(S(X))$ such that $D^{-1}(f) \subseteq B(x, \e/4)$. As $f \in M_{\e/4, \delta}(X^*)$, $d_1^*(f, \delta) <\e/4$. By Lemma~\ref{lem3}, $d_2^*(f, \alpha) < \e/2$ for $0 <\alpha < \e \delta/8$. That is, $diam(S(B(X), f, \alpha)) < \e/2$.

Also, $D^{-1}(f) \subseteq B(x, \e/4)$ and $f \in D(S(X))$. Hence for any $y \in S(B(X), f, \alpha)$ and $z \in D^{-1}(f)$
\[
\|x-y\| \leq \|x - z\| + \|y - z\| \leq \e/4 + \e/2 < \e.
\]
Therefore,
\[
S(B(X), f, \alpha) \subseteq B(x, \e).
\]

$(c) \Ra (d)$ is obvious.

$(d) \Ra(c)$. If $(c)$ does not hold, there exists $\e > 0$ such that for all $\delta > 0$, there exists $x_\delta \in S(X)$ such that for all $f \in D(S(X))$, there exists $g_f \in D(S(X))$ and $y_{f} \in D^{-1}(g_f)$ such that $\|f - g_{f}\| < \delta/2$ and $\|x_\delta - y_{f}\|\geq \e$.

Define $\Delta = \{g_{f} : f \in D(S(X))\}$. Clearly, $\Delta$ forms a $\delta/2$-net in $D(S(X))$. Since $D(S(X))$ is dense in $S(X^*)$, $\Delta$ forms a $\delta$-net in $S(X^*)$ as well. Define a support mapping $\Phi : D(S(X)) \to S(X)$ as follows~: \blr
\item If $g \in \Delta$, we choose any $f \in D(S(X))$ such that $g = g_f$ and define $\Phi(g) = y_f$. Observe that, by definition of $\Delta$, there is at least one such $f \in D(S(X))$.
\item If $g \in D(S(X)) \setminus \Delta$, define $\Phi(g) = y$ for some $y \in D^{-1}(g)$.
\el
Now, as $\|x_\delta - \Phi(g)\|\geq \e$ for all $g \in \Delta$, $\Phi(\Delta)$ does not form an $\e$-net in $S(X)$. So, $(d)$ does not hold.
\end{proof}

\bibliographystyle{amsplain}

\end{document}